\documentclass[11pt,oneside,english]{amsart}
  \usepackage{amsmath}
  \usepackage{amsfonts}
  \usepackage{amsthm}
  \usepackage{amssymb}
  \usepackage[all]{xy}
  \usepackage{graphicx}
  \usepackage{verbatim}   
  \usepackage{longtable}
  \usepackage{stmaryrd}   
  \usepackage{mathrsfs}   
  \usepackage{color}
  \usepackage{enumerate}
  \usepackage[center,loose,nooneline]{subfigure}
 \usepackage{bm}
 \usepackage{float}
 \usepackage{afterpage}

\makeatletter \theoremstyle{plain}
\newtheorem{thm}{Theorem}[section]
\newtheorem{lem}[thm]{Lemma}
\newtheorem{cor}[thm]{Corollary}
\newtheorem{prop}[thm]{Proposition}

\theoremstyle{remark}
\newtheorem{rmk}[thm]{Remark}
\numberwithin{equation}{section}

\theoremstyle{definition}
\newtheorem{defin}[thm]{Definition}
\newtheorem{example}[thm]{Example}
\newtheorem{question}[thm]{Problem}

 \newcommand{\N}{\mathbb{N}} 
 \newcommand{\Q}{\mathbb{Q}}
 \newcommand{\R}{\mathbb{R}}
 
\newcommand{\C}{\mathbb{C}} 
 
 \newcommand{\sphere}{\mathbb{S}}

\newcommand{\Hhh}{\mathscr{H}}

 \newcommand{\diam}{\mathrm{diam}}
 
 \newcommand{\simd}{\mathrm{d}}
 \newcommand{\dm}{\mathrm{d}}

  \newcommand{\0}{\mathbf{0}}

\newcommand{\Ss}{\mathcal{S}}  
  
\newcommand{\s}{\mathbf{s}}

 \newcommand{\lm}{\varnothing}

\newcommand\id{\operatorname{id}}

\newcommand{\ep}{\epsilon}

\newcommand{\reals}{\mathbb{R}}

\newcommand{\nats}{\mathbb{N}}
\newcommand{\comps}{\mathbb{C}}

\newcommand{\inv}{^{-1}}

\newcommand{\ovl}{\overline}

\newcommand{\lan}{\langle}
\newcommand{\ran}{\rangle}

\newcommand{\til}{\widetilde}

\newcommand{\Hdim}{\Hhh}

\newcommand{\subeq}{\subseteq}

\newcommand{\bslash}{\backslash}

\newcommand{\mbf}{\mathbf}

\def\diam{\operatorname{diam}}

\def\im{\operatorname{im}}

\begin{document}
\title{Iterated function system quasiarcs}
\author{Annina Iseli, Kevin Wildrick}
\address {Department of Mathematical Sciences, Montana State University, P.O.\ Box 172400
Bozeman, MT 59717 USA}
\email{kevin.wildrick@montana.edu}
\address {Mathematisches Institut,
Universit\"at Bern,
Sidlerstrasse 5,
CH-3012 Bern,
Switzerland}
\email{annina.iseli@math.unibe.ch}

\begin{abstract}
We consider a class of iterated function systems (IFSs) of contracting similarities of~$\reals^n$, introduced by Hutchinson, for which the invariant set possesses a natural H\"older continuous parameterization by the unit interval. When such an invariant set is homeomorphic to an interval, we give necessary conditions in terms of the similarities alone for it to possess a quasisymmetric (and as a corollary, bi-H\"older) parameterization. We also give a related necessary condition for the invariant set of such an IFS to be homeomorphic to an interval. 
\end{abstract}
\maketitle
\setcounter{equation}{0}

\section{Introduction}

\thispagestyle{empty}

Consider an iterated function system (IFS) of contracting similarities 
$\Ss=\{S_{1}, ... , S_{N}\}$ of $\reals^n$, $N \geq 2$, $n\geq 1$. For $i=1,\hdots, N$, we will denote the scaling ratio of $S_i:\R^n\to\R^n$ by $0<r_i<1$. In a brief remark in his influential work \cite{hutchinson1981}, Hutchinson introduced a class of such IFSs for which the invariant set is a Peano continuum, i.e., it possesses a continuous parameterization by the unit interval. There is a natural choice for this parameterization, which we call the \emph{Hutchinson parameterization}.

\begin{defin}\label{def_fraccurve}\textbf{[Hutchinson, 1981]} The pair $(\Ss,\gamma)$, where $\gamma$ is the invariant set of an IFS $\Ss=\{S_{1},...,S_{N}\}$ with scaling ratio list $\{r_{1},...,r_{N}\}$ is said to be an \emph{IFS path}, if there exist points $a,b\in\R^{n}$ such that 
\begin{enumerate}[(i)]
\item $S_1(a)=a$ and $S_{N}(b)=b$,
\item $S_{i}(b)=S_{i+1}(a)$, for any $i\in\{1,...,N-1\}.$
\end{enumerate} 
\end{defin} 
Recall that the \emph{Hutchinson operator} $T$ associated to an IFS is defined by
$T(A)=\bigcup_{i=1}^{N}S_{i}(A)$ for sets $A\subseteq\R^n$. For an IFS path, the image of the line segment connecting $a$ and $b$ under any iterate $T^k$, $k \in \nats$, is a connected, piecewise linear set with a natural parameterization arising from the IFS~$\Ss$. As $k$ tends to infinity, these parameterizations converge to the Hutchinson parameterization of the invariant set (see Section~\ref{sec_param}).

The canonical example of an IFS path is the Koch snowflake arc; however, in general the invariant set of an IFS path need not be homeomorphic to the unit interval (e.g., the Peano space-filling arc). Characterizing the IFS paths for which this is the case seems to be a very difficult problem that displays chaotic behavior. Since the invariant set of an IFS does not determine the IFS, it seems that the most reasonable version of this task is to characterize the IFS paths for which the Hutchinson parameterization is injective (and hence a homeomorphism). As will be shown in Section \ref{sec_param} (see in particular Proposition~\ref{prop_arc_iff_inj}), the injectivity of the Hutchinson parameterization can be easily reformulated in terms of the pair $(\Ss, \gamma)$ as follows:

\begin{defin}\label{def_inj}
We say that an IFS path $(\Ss,\gamma)$ is an \emph{IFS arc} if
\begin{enumerate}[(i)]
\item $S_{i}(\gamma)\cap S_{i+1}(\gamma)=\{S_{i+1}(a)\}$ for $i\in\{1,...,N-1\},$
\item $S_i(\gamma)\cap S_j(\gamma)=\lm$ for $i,j\in \{1,...,N-1\}$ with $|i-j| > 1$,
\end{enumerate} 
where $a \in \reals^n$ is as in Definition \ref{def_fraccurve}.
\end{defin}

In many concrete examples, the invariant set $\gamma$ of an IFS arc is known to be a \emph{quasiarc}, i.e., a quasisymmetric image of $[0,1]$. For an excellent introduction to quasisymmetric mappings, see the foundational article of Tukia and V\"ais\"al\"a \cite{TV} and the book \cite{LAMS}. An arc in $\reals^2$ is a quasiarc if and only if the arc is the image of a compact interval under a quasiconformal self-mapping of $\reals^2$. In fact quasiarcs admit \emph{many} characterizations \cite{Ubiq}, including by a simple geometric condition called \emph{bounded turning}; see Section~\ref{sec_bt_notation} for details. Let us say that an IFS arc is an \emph{IFS quasiarc} if its invariant set is a quasiarc.  

IFS quasiarcs play an important role in the theory of quasiconformal mappings, particularly with respect to questions about dimension distortion, and also appear in connection with Schwarzian rigid domains \cite{Zippers}. For this reason, it is desirable to have a large, concrete family of IFS quasiarcs available for study and use. To those familiar with the theory of quasiconformal mappings, it may seem likely that \emph{every} IFS arc is an IFS quasiarc due to the apparent self-similarity of the construction. However, this is not the case \cite{Wen}:

\begin{thm}[Wen and Xi, 2003]\label{thm_counter} There is an IFS arc that is not a quasiarc. 
\end{thm}

Roughly speaking, the fact that the invariant set of an IFS arc is indeed an arc indicates that any obstruction preventing it from being a quasiarc can only occur infinitesimally and not globally (see Section~\ref{sec_FC_with_BT}); this seems impossible for a self-similar object.  However, when the ratios $r_1$ and $r_N$ are not equal, one is not \emph{a priori} guaranteed that scaled copies of small pieces of the IFS arc also appear at large scale. In fact, this is the only obstacle, as we indicate in Section \ref{sec_FM}. Our main result gives a fairly large class of IFS arcs in $\reals^n$ for which this obstacle does not occur. Aside from the assumption that the IFS path is an arc, the class is defined in terms of the similarities alone and the condition defining the class is simple to check. 

Since the property of being a quasiarc in $\reals^n$ is invariant under similarities, it is natural to assume that an IFS path ($\Ss,\gamma)$ in $\reals^n$ is \emph{normalized} so that the point $a$ fixed by $S_1$ is the origin and the point $b$ fixed by $S_N$ is $\mathbf{e}_1 = (1,0,...,0)\in\R^n$. We recall that for each (contracting) similarity $S_i:\R^n\to\R^n$ of an IFS path $(\Ss,\gamma)$ there exists a ratio $r_i\in(0,1)$, an orthogonal transformation $A_i$, and a translation vector $b_i\in\R^n$ such that $S_i(x)=r_iA_i(x)+b_i$ for all $x\in\R^n$. Note that if $(\Ss,\gamma)$ is a normalized IFS path, then $b_1=0$.

\begin{thm}\label{thm_BT_intro}
Let $(\Ss,\gamma)$ be a normalized IFS arc in $\R^n$ such that:
\begin{enumerate}[(A)]
\item\label{ratios} There exist numbers $t,s\in\N$ such that $r_1^t = r_N^s$.
\item\label{angles}  Either $A_1^t=A_N^s$, or, there exist numbers $p,q\in\N$ such that $A_1^q=A_N^p=\mathrm{Id}$, where $\mathrm{Id}$ denotes the identity matrix.
\end{enumerate}
Then $(\Ss, \gamma)$ is an IFS quasiarc. 
\end{thm}

In Corollary~\ref{cor_BT_intro_1}, we will see that if $n=2$, then one can find plenty of easy-to-check conditions that are sufficient for Theorem~\ref{thm_BT_intro} to hold.

Condition \eqref{ratios} is violated by the example of Wen and Xi, and hence Theorem~\ref{thm_BT_intro} fails without it. We do not know if Theorem~\ref{thm_BT_intro} fails without condition \eqref{angles}. However, we suspect that if condition~(B) fails badly (see e.g.\ condition (2) in Theorem~\ref{thm_non-inj}), then the path $\gamma$ has self-intersections and hence $ (\Ss,\gamma) $ is not an IFS arc. Theorem~\ref{thm_non-inj} proves this conjecture for some cases of IFS paths in $\R^2$.
 
Besides the algebraic conditions of Theorem~\ref{thm_BT_intro}, we also briefly discuss a simpler and well-known condition for an IFS arc to be an IFS quasiarc, which we call the \emph{cone containment condition} (see also \cite{Wen}). The cone containment condition is harder to check but covers many cases not covered by Theorem \ref{thm_BT_intro}. However, not all IFS arcs satisfying the hypotheses of Theorem \ref{thm_BT_intro} satisfy the cone containment condition. Also note that neither the conditions in Theorem~\ref{thm_BT_intro} nor the cone containement condition guarantee that an IFS path $(\Ss, \gamma)$ is actually a topological interval (see Remark~\ref{rmk_sierpin}). 

A similar but slightly larger class of iterated function systems, called \emph{zippers}, has been examined by Aseev, Tetenov, and Kravchenko in \cite{Aseev}. There, the authors give a different collection of conditions on the IFS that guarantee that the invariant set is a quasiarc; there seems to be no overlap between those results and Theorem \ref{thm_BT_intro}. Other subclasses of zippers have been considered in connection with quantitative dimension distortion of planar quasiconformal mappings; in this case holomorphic motions can be used to show the quasiarc property \cite{AstalaChat}. A wonderful visualization for such function systems can be found at \cite{web}. In neither of these works are the zippers considered actually IFS paths, although in some cases the invariant set of the zipper can also be realized as the invariant set of an IFS path. 

Theorem \ref{thm_BT_intro} provides a large, concrete class of IFS quasiarcs. We also show that IFS quasiarcs have special properties not shared by all quasiarcs nor by all IFS arcs. For example, IFS quasiarcs have particularly nice parameterizations: 

\begin{thm}\label{thm_bi}
Let $(\Ss,\gamma)$ be an IFS arc with similarity dimension $s$.  Then $(\Ss, \gamma)$ is an IFS quasiarc if and only if the Hutchinson parameterization of $\gamma$ by the interval $[0,1]$ is $\frac{1}{s}$-bi-H\"older continuous. 
\end{thm}

We recall that a metric space $(X,d)$ is Ahlfors $s$-regular, $s \geq 0$, if there is a constant $K \geq 1$ such that for each $x \in X$ and $0 < r < 2\diam X$,
$$\frac{r^s}{K} \leq \Hhh^s(B(x,r)) \leq Kr^s,$$
where $\Hhh^s$ denotes $s$-dimensional Hausdorff measure. A $\frac{1}{s}$-bi-H\"older continuous image of $[0,1]$ is Ahlfors $s$-regular, and so Theorem~\ref{thm_bi} implies that if $(\Ss,\gamma)$ is an IFS quasiarc with similarity dimension $s$, then $\gamma$ is Ahlfors $s$-regular. In particular, the similarity dimension $s$ of $(\Ss,\gamma)$ is equal to the Hausdorff dimension of $\gamma$. 

An immediate consequence of Theorem \ref{thm_bi} is the following somewhat surprising statement:
\begin{cor}\label{bi-Lip equiv} Let $\alpha$ and $\beta$ be IFS quasiarcs with equal similarity dimension. Then $\alpha$ and $\beta$ are bi-Lipschitz equivalent. 
\end{cor}

Statements such as Corollary \ref{bi-Lip equiv} are referred to in the literature as \emph{Lipschitz equivalence} results and are ubiquitous; see, for example,  \cite{falconermarsh1989}, \cite{falconermarsh1992}, \cite{self-conformal}, \cite{Llorente}, as well as \cite{broken} and \cite{survey} and the references therein. 

The surprising nature of Corollary \ref{bi-Lip equiv} is illustrated by the fact that the four arcs depicted in Figure \ref{pic_4frac} represent bi-Lipschitz equivalent quasiarcs (metrized as subsets of $\reals^2$) of similarity and Hausdorff dimension $s=\frac{\log 4}{\log 3}$; see Sections \ref{sec_cone} and \ref{sec_ex} for the precise definitions of these arcs. 

As pointed out by Wen \cite{Wen}, Theorem~\ref{thm_counter} implies that Corollary~\ref{bi-Lip equiv} fails for the class of IFS arcs. Moreover, it is not difficult to find (non-IFS) quasiarcs for which the statement fails for Hausdorff dimension. 

While Theorem \ref{thm_bi} indicates that the class of IFS quasiarcs is much smaller than the class of all quasiarcs, an important result of Rohde \cite{Rohde} and its generalization by Herron and Meyer \cite{HerronMeyer} indicates that \emph{every} quasiarc can be obtained, up to a bi-Lipschitz mapping, by a snowflake-type construction.

One may (reasonably) complain that it is difficult to know if the invariant set of an IFS path is an arc. However, there are criteria, similar in spirit to those of Theorem \ref{thm_BT_intro}, that guarantee an IFS path in $\R^2$ is \emph{not} an arc. 
The angles $\alpha_1$ and $\alpha_N$ appearing in the statement below are the angles of counterclockwise rotation provided by the orthogonal transformations $A_1$ and $A_N$ of $\R^2$ associated to $S_1$ and $S_N$.

\begin{thm}\label{thm_non-inj}
Let $(\Ss,\gamma)$ be a normalized IFS path in $\R^2$ such that
\begin{enumerate}
\item $S_1$ and $S_N$ are orientation preserving,
\item there exist numbers $t,s\in\N$ such that $r_1^t=r_N^s$, and 
 $$(t\alpha_1 -s\alpha_N) \in 2\pi(\R\backslash\Q),$$
\item there exists $i\in\{1,...,N-1\}$ such that the following conditions hold: 
\begin{enumerate}[a)]
\item $S_i$ and $S_{i+1}$ are either both orientation preserving or both orientation reversing,
\item the set 
$$\{\theta\in[0,2\pi):\ [S_i(\gamma)\cap R_\theta(S_{i+1}(\gamma))]\backslash\{z\} \neq \emptyset\}$$
contains an open interval, where $R_\theta$ denotes the counterclockwise rotation by $\theta$ around the point $z:=S_i(\bm{e}_1)=S_{i+1}(\0)$.
\end{enumerate}
\end{enumerate}
Then $\gamma$ is not an IFS arc. 
\end{thm}

The proof of Theorem \ref{thm_non-inj} will show that many variants of this result are possible. Condition (3) is the only assumption that requires \emph{a priori} knowledge of $\gamma$, and indeed we suspect that condition (3) holds for any IFS path $(\Ss, \gamma)$ for which $\gamma$ is not a line segment.

We will give preliminary definitions in Section~\ref{sec_defs}. In Section~\ref{sec_estimates}, we develop a key estimate, which will be used repeatedly. Section~\ref{sec_FC_with_BT} gives the proof of Theorem~\ref{thm_BT_intro} and of Theorem~\ref{thm_bi}, and Section~\ref{sec_non_inj} gives the proof of Theorem~\ref{thm_non-inj}. We conclude with a discussion of some open problems.

This work arose from the Master's thesis of the first author. We wish to thank her advisor Zolt\'an Balogh for his kind guidance and substantial input. We also wish to thank Kari Astala and Sebastian Baader for valuable discussions. Furthermore, we extend our thanks to the two referees for carefully reading our paper and providing helpful remarks and suggestions. 
\afterpage{
\begin{center}
\begin{figure}[H]
\includegraphics{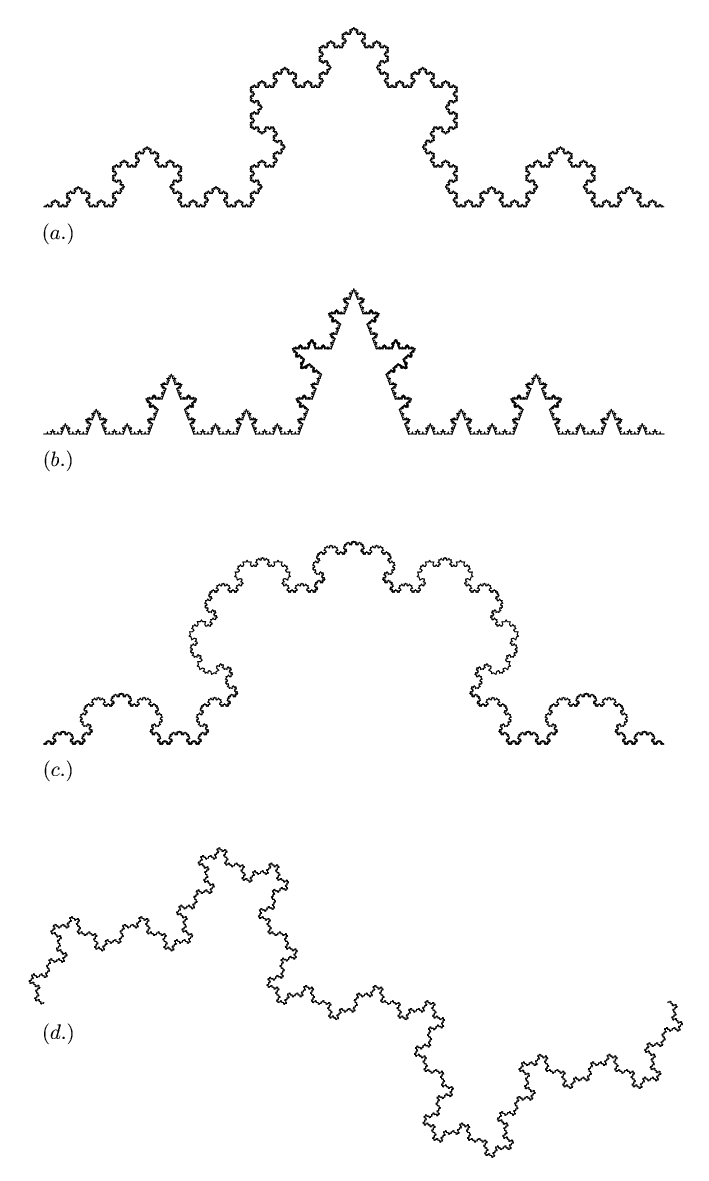}
\caption{Four examples of an IFS quasiarc of similarity and Hausdorff dimension $\frac{\log(4)}{\log(3)}$.}
\label{pic_4frac}
\end{figure}
\end{center}}

\section{Background and Notation}\label{sec_defs}

We will denote by $\dm$ the standard distance on $\R^n$ for $n\geq 2$, by $|a-b|$ the Euclidean distance between points $a,b\in \R$, and by $\diam(E)$ the diameter of a set $E$ with respect to the metric considered on the ambient space of $E$. 
Furthermore, for a $(n\times n)$-matrix $A$ and a $k\in\N$, we will write $A^k$ for the matrix product $A....A$ ($k$-times), and we interpret $A^0=\mathrm{Id}$, where $\mathrm{Id}$ is the identity matrix.\\

A compact non-empty set $K$ is called the invariant set of an IFS $\Ss$ if $T(K)=K$, where $T$ denotes the Hutchinson operator associated to $\Ss$. Each IFS $\Ss$ of contracting similarities of $\reals^n$ admits a unique invariant set. Moreover, it can be constructed as a limit: Let $K_0$ be any non-empty compact subset of $\R^n$. Define $K_k$ for $k\in\N$ recursively by $K_k=T(K_{k-1})$. Then $K$ is the limit of the sequence $\{K_k\}_{k\in\N}$ with respect to the Hausdorff distance (which is a metric on the space of non-empty compact subsets of $\R^n$). 

Let $K$ be the invariant set of an IFS $\Ss$ of contracting similarities in $\R^n$. We call the solution $s>0$ of the equation 
$$\sum_{i=1}^Nr_i^s=1$$
the \emph{similarity dimension} of $K$ (we assume throughout that $N \geq 2$).  Note that for any invariant set $K$ of an IFS $\Ss$, it holds that $\dim K \leq s$, where $\dim K$ denotes the Hausdorff dimension of $K$. However, the inequality $\dim K \geq s$ only holds under stronger assumptions on the IFS.

\subsection{Basic Notation for IFS paths} 
Throughout, we will assume that all IFS paths are normalized as described in the introduction.  We denote the line segment between $\0$ and the point $\bm{e}_1=(1,0,\hdots,0)$ in $\R^n$ by $I$. If $T$ is the Hutchinson operator associated to an IFS path $(\Ss, \gamma)$, then the approximation of $\gamma$ by the iterates $\{T^k(I)\}_{k \in\nats}$ is an approximation by piecewise linear paths connecting $\0$ and $\bm{e}_1$.

For a sequence $\sigma=(\sigma_1,\sigma_2,...,\sigma_m)\in\{1,...,N\}^m$ of length $m\in\N$, we will write $S_\sigma$ for the composition of maps
 $S_{\sigma_1}\circ S_{\sigma_2}\circ ...\circ S_{\sigma_m}$. Analogously, we will write $r_\sigma$ for the product of the ratios $r_{\sigma_1} r_{\sigma_2}... r_{\sigma_m}$.
We call $p\in \gamma$ a \emph{vertex of generation m}, if $p=S_\sigma(\0)$ or $p=S_\sigma(\bm{e}_1)$ for some $\sigma\in\{1,...,N\}^m$.  We will call a set $S_\sigma(\gamma)$ a \emph{copy of $\gamma$ of generation $m$}.
Note that each vertex of generation $m$ of $\gamma$, is also a vertex of generation $m'$ for any $m'>m$. Similarly, each copy of $\gamma$ of generation $m'$ is contained in a copy of $\gamma$ of generation $m$, for each $m'>m$.

In Proposition \ref{prop_arc_iff_inj} we will show, independently of the rest of the results in this article, that the Hutchinson parameterization of an IFS arc $(\Ss, \gamma)$ is injective, and hence that the invariant set is homeomorphic to an interval. This result will be used \emph{throughout the entire article}. In particular, we will use the fact that if $(\Ss, \gamma)$ is an IFS arc, then there is a natural ordering $\leq$ on $\gamma$ that is determined by declaring that $\0 < \bm{e}_1$. For a set $A\subseteq \gamma$, and a point $x \in \gamma$, we will write $x\leq A$ if $x\leq y$ for any $y\in A$. By $\gamma_{x,y}$ we denote the arc that connects $x$ to $y$ within $\gamma$, i.e. if $x \leq y$, then $\gamma_{x,y}=\{p\in\gamma: x\leq p \leq y\}.$

\subsection{Quasiarcs and bounded turning}\label{sec_bt_notation}
A homeomorphism $f \colon (X,d_X) \to (Y,d_Y)$ between metric spaces is \emph{quasisymmetric} if there is a homeomorphism $\eta \colon [0,\infty) \to [0,\infty)$ such that for all triples $a,b,c$ of distinct points in $X$, 
$$\frac{d_Y(f(a),f(b))}{d_Y(f(a),f(c))} \leq \eta \left( \frac{d_X(a,b)}{d_X(a,c)}\right).$$
Two spaces are \emph{quasisymmetrically equivalent} if there is a quasisymmetric homeomorphism between them. A fundamental theorem of Tukia and V\"ais\"al\"a \cite{TV} states that a metric space $(X,d)$ is quasisymmetrically equivalent to the interval $[0,1]$ if and only if 
\begin{itemize}
\item $(X,d)$ is homeomorphic to $[0,1]$,
\item $(X,d)$ is \emph{doubling}, i.e., there is a constant $D \geq 1$ such that for all $r>0$, each ball of radius $r$ in $X$ can be covered by at most $D$ balls of radius $r/2$, and
\item $(X,d)$ has \emph{bounded turning}, i.e., there is a constant $C \geq 1$ such that given distinct points $a,b \in X$, there is a continuum $E\subeq X$ containing $a$ and $b$ satisfying 
$$\diam E \leq Cd(a,b).$$
\end{itemize}
As mentioned in the introduction, a metric space that is quasisymmetrically equivalent to $[0,1]$ is called a \emph{quasiarc}. We will prove that the invariant sets of certain IFS arcs in $\reals^n$ are quasiarcs by verifying the bounded turning condition, since subsets of $\reals^n$ are always doubling. 

\section{Fundamental estimates}\label{sec_estimates}
We now provide a collection of key estimates that will be repeatedly used in the proofs of Theorem~\ref{thm_BT_intro} (in order to show bounded turning) and Theorem~\ref{thm_bi}. 

Let $(\Ss,\gamma)$ be an IFS arc in $\reals^n$. Define 
\begin{equation}
D_{\Ss}:=\min_{z,\tilde{z}} \min \{\dm(x,y)\ :\ x\leq \gamma_{z,\tilde{z}}\leq y, \ x,y\in\gamma\},
\end{equation}
where the first minimum is taken over pairs of distinct generation $1$ vertices $z$ and $\tilde{z}$. By Proposition~\ref{prop_arc_iff_inj}, the set $\gamma$ is an arc, and so $D_{\Ss} > 0$.

Fix $x,y\in\gamma$ with $x< y$. We will estimate $\dm(x,y)$ as well as $\diam(\gamma_{x,y})$ using the similarities in $\Ss$. Towards this end, let $m$ be the smallest number in $\N$ for which there exists a generation $m$ vertex $z$ satisfying $x\leq z\leq y$. By definition, there is a sequence $\sigma\in \{1,...,N\}^{m-1}$ such that $\gamma_{x,y}\subseteq S_\sigma(\gamma)$. 

\begin{rmk} The definition of $m$ has an (imperfect) analogue in Gromov hyperbolic geometry. Consider a tree whose vertices are labelled by finite sequences with terms in $\{1,\hdots,N\}$ and where two vertices are connected by an edge if one of the corresponding sequences extends the other. If this tree is equipped with the usual graph distance, it becomes Gromov hyperbolic, and its boundary at infinity is a Cantor set that can be labeled by infinite sequences with terms in $\{1,\hdots, N\}$ and hence maps surjectively but not injectively to $\gamma		$. Given two points $\sigma$ and $\tau$ in this boundary, their \emph{Gromov product} is the length of the initial sequence they share.

This analogy indicates another approach to the topic of this paper. One could study the quasisymmetry class of an IFS path $\gamma$ via quasi-isometries of its hyperbolic filling, which is a Gromov hyperbolic space whose boundary at infinity is \emph{precisely} $\gamma$. 
\end{rmk}

We prove basic estimates in two different cases. 

\smallskip

\underline{Case 1}: Assume that there exists another generation $m$ vertex $\tilde{z}\neq z$ with $x \leq \tilde{z} \leq y$. We call the pair $(x,y)$ a \emph{case-1-pair}. We may assume without loss of generality that $z<\tilde{z}$. This implies that $x\leq \gamma_{z,\tilde{z}}\leq y$. Applying the similarity $S_\sigma^{-1}$ yields 
$S_\sigma^{-1}(x)\leq S_\sigma^{-1}(\gamma_{z,\tilde{z}})\leq S_\sigma^{-1}(y).$
Both $S_\sigma^{-1}(z)$ and $S_\sigma^{-1}(\tilde{z} )$ are vertices of generation $1$, and 
$S_\sigma^{-1}(\gamma_{z,\tilde{z}})=\gamma_{S_\sigma^{-1}(z),S_\sigma^{-1}(\tilde{z})}.$
By definition, we obtain $\dm(S_\sigma^{-1}(x), S_\sigma^{-1}(y))\geq D_{\Ss}$ and therefore \begin{equation}\label{case1_1}
\dm(x,y)\geq r_\sigma D_{\Ss} .
\end{equation}
On the other hand, $\gamma_{x,y}\subseteq S_\sigma(\gamma)$, so
\begin{equation}\label{case1_2}
\diam{\gamma_{x,y}}\leq \diam(\gamma)r_\sigma.
\end{equation}
From \eqref{case1_1} and \eqref{case1_2} it follows that \begin{equation}\label{case1_3}
\diam{\gamma_{x,y}}\leq \frac{\diam(\gamma)}{D_{\Ss}}\dm(x,y).
\end{equation} Note that in particular this verifies the bounded turning condition for the points $x$ and $y$.

\smallskip

\underline{Case 2}: Assume now that $z$ is the only generation $m$ vertex such that $x \leq z \leq y$. We call the pair $(x,y)$ a \emph{case-2-pair} and the triple $(x,z,y)$ a \emph{case-2-triple}. It follows that $x$ and $y$ are contained in adjacent copies of $\gamma$ of generation $m$.
In particular, by definition, $z$ is the right endpoint of the copy of $\gamma$ of generation $m$ containing $x$ and it is the left end point of the copy of $\gamma$ of generation $m$ containing $y$.
Thus, there exists a number $\sigma_m\in\{1,...,N-1\}$ such that $\gamma_{x,z}\subset S_{\sigma\sigma_m}(\gamma)$ and $\gamma_{z,y}\subset S_{\sigma(\sigma_m+1)}(\gamma)$. Moreover, 
\begin{equation}\label{choose_sigma_m} S_{\sigma\sigma_m}(\bm{e}_1)=z=S_{\sigma(\sigma_m+1)}(\0).\end{equation}

Assume for the moment that $x\neq z$. Let $k\in\N$ and consider all the copies of $\gamma$ of generation $m+k$ that are contained in $S_{\sigma\sigma_m}(\gamma)$. Note that $z$ is the right endpoint of the right-most of these copies, i.e., 
$$z=S_{\sigma\sigma_m \underbrace{N\hdots N}_{k\ \text{times}}}(\bm{e}_1).$$
We will abuse notation here and later in similar situations by writing 
$\sigma\sigma_m N^{k}:=\sigma\sigma_m \underbrace{N\hdots N}_{k\ \text{times}}.$ 
Note that 
\begin{enumerate}[(i)]
\item for $k=0$, it trivially holds that 
$x\in S_{\sigma\sigma_m}(\gamma)= S_{\sigma\sigma_m N^{k}}(\gamma),$
\item for any $k\in\N$, $S_{\sigma \sigma_m N^{k+1}}(\gamma)\subset S_{\sigma\sigma_m N^{k}}(\gamma)$, and 
\item $\lim_{k\shortrightarrow \infty} \diam (S_{\sigma\sigma_m N^{k}}(\gamma))=0$ and thus, since $x\neq z$, there will be some $k\in\N$ for which 
$x\notin S_{\sigma(\sigma_m)N^{k}}(\gamma).$
\end{enumerate}
By (i), (ii) and (iii), there exists  $k\in \N$ such that 
\begin{equation}\label{choose k} x\in S_{\sigma\sigma_m N^{k}}(\gamma) \ \text{and}\ x\notin  S_{\sigma(\sigma_m)N^{k+1}}(\gamma).
\end{equation} 
Set $\tilde{z}=S_{\sigma \sigma_m N^{k+1}}(\0) $. Thus $z$ and $\tilde{z}$ are distinct vertices of generation $m+k+1$ separating $x$ and $y$. Therefore, analogously to Case 1, it follows that:
\begin{equation}\label{case2_1_x}
\dm(x,z)\geq r_\sigma r_{\sigma_m} r_N^k D_{\Ss} ,
\end{equation}
as well as
\begin{equation}\label{case2_2_x}
\diam{\gamma_{x.z}}\leq \diam(\gamma)r_\sigma r_{\sigma_m} r_N^k.
\end{equation}
Hence, \begin{equation}\label{case2_3_x}
\diam{\gamma_{x,z}}\leq \frac{\diam(\gamma)}{D_{\Ss}}\dm(x,z).
\end{equation}
Note that \eqref{case2_3_x} also holds trivially if $x=z$. In either situation, if it happens that $y = z$, then \eqref{case2_3_x} verifies the bounded turning condition for the points $x$ and $y$. 
with the same constant as in Case 1. 

If $y\neq z$, an analogous argument shows the existence of $l \in \nats$ such that 
\begin{equation}\label{choose l} y\in S_{\sigma(\sigma_m+1)1^{l}}(\gamma) \ \text{and}\ y\notin  S_{\sigma(\sigma_m+1)1^{l+1}}(\gamma).
\end{equation} 
As above, it follows that 

 \begin{equation}\label{case2_3_y}
\diam{\gamma_{z,y}}\leq \frac{\diam(\gamma)}{D_{\Ss}}\dm(z,y).
\end{equation}
As before, 
\eqref{case2_3_y} also holds trivially if $y=z$, and in either situation, 
if it happens that $x = z$, then \eqref{case2_3_x} verifies the bounded turning condition for the points $x$ and $y$ with the same constant as in Case 1. 

\medskip

A consequence of these estimates is that verifying the bounded turning condition only requires estimating distances (not diameters) between the points in case-2-triples. 

\begin{lem}\label{lem_3pt_BT} Let $(\Ss, \gamma)$ be an IFS arc in $\reals^n$. Suppose that there is a constant $C \geq 1$ such that for all case-2-triples $(x,z,y)$, 
\begin{equation}\label{points} \max\{\dm(x,z), \dm(z,y)\} \leq C \dm(x,y).\end{equation} 
Then $\gamma$ is of bounded turning with constant $\frac{2C\diam(\gamma)}{D_{\Ss}}.$  

Conversely, if $\gamma$ is an IFS arc of bounded turning with constant $C \geq 1$, then \eqref{points} holds for all points $x \leq z \leq y$ of $\gamma$. 
\end{lem}

\begin{proof}
By the arguments leading to \eqref{case1_3}, \eqref{case2_3_x}, and \eqref{case2_3_y} it suffices to verify the bounded turning condition for a case-2-pair $(x,y)$ such that the corresponding case-2-triple $(x,z,y)$ satisfies $x\neq z \neq y$.
By \eqref{case2_3_x}, \eqref{case2_3_y}, and \eqref{points}, we see that
\begin{equation*}
\diam(\gamma_{x,y})\leq \diam (\gamma_{x,z})+\diam (\gamma_{z,y})\leq \frac{\diam(\gamma)}{D_{\Ss}}(\dm(x,z)+\dm(z,y))\leq \frac{2C\diam(\gamma)}{D_{\Ss}}\dm(x,y),
\end{equation*}
as desired.

The proof of the converse statement is a direct consequence of the definitions. 
\end{proof}

\section{IFS arcs with bounded turning}\label{sec_FC_with_BT}

\subsection{The proof of Theorem \ref{thm_BT_intro}}\label{sec_proof_of_BT_thm}
A metric arc is a metric space that is homeomorphic to the interval $[0,1]$. We begin by showing that the obstacles preventing any metric arc from being of bounded turning must occur infinitesimally.  

\begin{lem}\label{lem1_3pt}
Let $(\gamma, d)$ be an metric arc equipped with an ordering $\leq$ compatible with its topology.
Let $(x_i,z_i,y_i)_{i\in\N}$ be a sequence of triples of points in $\gamma$ with $x_i \leq z_i \leq y_i$ and $x_i\neq y_i$ for each $i \in \nats$.  If  \begin{equation}\label{eq10}
\lim_{i\shortrightarrow\infty} \left(\frac{d (x_i,y_i)}{\max\{d (x_i,z_i),d (z_i,y_i)\}} \right)= 0,
\end{equation}
then
\begin{equation}\label{eq12}
\lim_{i\shortrightarrow \infty} \max\{d(x_i,z_i),d(z_i,y_i)\}  =0
\end{equation}
\end{lem} 

\begin{proof}
It follows from the compactness of $\gamma$ alone that \eqref{eq10} implies 
\begin{equation}\label{eq10'}\lim_{i \to \infty} d(x_i,y_i) = 0.\end{equation} 
Since $\gamma$ is a metric arc, it has a parameterization by $[0,1]$ that is uniformly continuous and has uniformly continuous inverse. Hence, \eqref{eq10'} implies that the diameter of the arc between $x_i$ and $y_i$ tends to $0$ as well, yielding \eqref{eq12}. 
\end{proof} 

\begin{rmk}\label{rmk_angles} We will not use condition (B) in the statement of Theorem \ref{thm_BT_intro} as it is stated. Instead, we employ the following fact: for two invertible matrices $A$ and $B$, the following statements are equivalent:\begin{itemize}
\item[(i)] There exist constants $p,q\in \N$ such that $A^q=B^p=\mathrm{Id}$.
\item[(ii)] There exist constants $p',q'\in \N$ such that $A^L=B^L$ for any $L,K\in\N$ that are multiples of $p'q'$.
\end{itemize}
Trivially, (i) implies (ii). The converse is an easy calculation.
\end{rmk}

\begin{proof}[Proof of Theorem \ref{thm_BT_intro}]
Let $(\Ss,\gamma)$ be an IFS arc in $\reals^n$ satisfying the hypotheses of Theorem~\ref{thm_BT_intro}. We claim that if $(x_i,z_i,y_i)_{i\in\N}$ is a sequence of case-2-triples in $\gamma$ that satisfies \eqref{eq10}, there exists another sequence of case-2-triples $(\tilde{x}_i,\tilde{z}_i,\tilde{y}_i)_{i\in\N}$ that satisfies \eqref{eq10} but not \eqref{eq12}. Given this claim, Lemma~\ref{lem1_3pt} implies that there is no sequence of case-2-triples satisfying \eqref{eq10}. Hence Lemma \ref{lem_3pt_BT} yields the desired conclusion.

To this end, let $(x,y,z)$ be a case-2-triple;  we may assume without loss of generality that $x\neq y \neq z$. Our claim will follow if we show that there exists another case-2-triple $(\tilde{x},\tilde{z},\tilde{y})$ satisfying the following two conditions:
\begin{eqnarray}
\label{cond_C}&&\max\{\dm(\tilde{x},\tilde{z}),\dm(\tilde{z},\tilde{y})\}\geq C,\text{ where $C>0$ only depends on $\gamma$,}\\
 &&\frac{\dm(x,y)}{\max\{\dm(x,z),\dm(z,y)\}} = \frac{\dm(\tilde{x},\tilde{y})}{\max\{\dm(\tilde{x},\tilde{z}),\dm(\tilde{z},\tilde{y})\}}.
\end{eqnarray}

In case $A_1^t=A_N^s$, set $M:= st\,$; in case $A_1^q=A_N^p=\mathrm{Id}$, set $M:=pqst$, where $q,p,s,t\in\N$ are the numbers given in the conditions of Theorem~\ref{thm_BT_intro}. 

As in Section~\ref{sec_estimates}, let $m$ be the smallest generation separating $x$ and $y$, let $\sigma \in \{1,\hdots,N\}^{m-1}$ be the sequence satisfying $\gamma_{x,y} \subeq S_\sigma(\gamma)$, and let $\sigma_m \in \{1,\hdots,N-1\}$ be the number such that 
$$ S_{\sigma\sigma_m}(\bm{e}_1)=z=S_{\sigma(\sigma_m+1)}(\0).$$ Furthermore, let $k$ and $l$ be numbers satisfying \eqref{choose k} and \eqref{choose l}, respectively; recall that these numbers locate $x$ and $y$ more precisely with respect to $z$ and the ratios $r_1$ and $r_N$. 

Let $\tilde{k}$ be the largest multiple of $M$ that is no greater than $k$, and $\tilde{l}$ be the largest multiple of $M$ that is no greater than $l$; note that $\til{k}$ and $\til{l}$ might be zero. Also, $k-\tilde{k}, l-\tilde{l}\in\{0,...,M-1\}$. Without loss of generality, we assume that $\tilde{k}t\leq \tilde{l}s$; the case $\tilde{k}t\geq \tilde{l}s$ is analogous.

Set $L=\frac{\tilde{k}t}{s}$ and $K=\til{k}$. Thus $L\leq \tilde{l}\leq l$ and $K=\tilde{k}\leq k$. Therefore $\gamma_{x,z}\subeq S_{\sigma\sigma_mN^K}(\gamma)$ and $\gamma_{z,y}\subeq  S_{\sigma(\sigma_m+1)1^L}(\gamma)$, and we can define a mapping $\psi: \gamma_{x,y}\rightarrow \gamma$  by  
$$\psi(u) = \begin{cases} S_{\sigma_m}\circ (S_{\sigma\sigma_mN^K})^{-1}(u) & u \in\gamma_{x,z}, \\
S_{\sigma_m+1}\circ (S_{\sigma(\sigma_m+1)1^L})^{-1}(u) & u \in \gamma_{z,y}.\\\end{cases}$$
 Set $\tilde{x}=\psi(x)$, $\tilde{y}=\psi(y)$ and $\tilde{z}=\psi(z)$. Note that thus $(\tilde{x},\tilde{z},\tilde{y})$ is a case-2-triple with $\til{x}\neq \til{z} \neq \til{y}$.
 Moreover:
\begin{itemize}
\item On $\gamma_{x,z}$, the mapping $\psi$ is a similarity with ratio $r_\sigma^{-1}r_N^{-K}$.
\item On $\gamma_{z,y}$, the mapping $\psi$ is a similarity with ratio $r_\sigma^{-1}r_1^{-L}$.
 \item It holds that $\gamma_{x,z}\cap\gamma_{z,y}=\{z\}.$
 \item By condition (A) and definition of $L$ and $L$ $r_1^{-L} = r_N^{-K}.$
\item  In case that $A_1^t=A_N^s$ (see condition (B)), we have $A_1^{L}=A_N^{N}$ and therefore the angle between $x$ and $y$ at $z$ equals the angle between $\tilde{x}$ and $\tilde{y}$ at $\tilde{z}$.
\item In the case when $A_1^q=A_N^p=\mathrm{Id}$: By Remark~\ref{rmk_angles}, since $K$ and $L$ are multiples of $pq$, the angle  between $x$ and $y$ at $z$ equals the angle between $\tilde{x}$ and $\tilde{y}$ at $\tilde{z}$.
\end{itemize}
From these observations, it follows that $\psi$ is a similarity of ratio $r_\sigma^{-1}r_N^{-K}$ on $\gamma_{x,y}$. So in particular, the assertion (ii) follows immediately. 

On the other hand, by \eqref{case2_1_x}, it follows that $\dm(x,z)\geq  r_\sigma r_{\sigma_m}r_N^k D_{\Ss}$ and thus 
\[\max\{\dm(\tilde{x},\tilde{z}),\dm(\tilde{z},\tilde{y})\}\geq \dm(\tilde{x},\tilde{z})=r_\sigma^{-1}r_N^{-K}\dm(x,z)\geq  r_{\sigma_m}r_N^{k-K}D_{\Ss}\geq D_{\Ss} \left(\min_{1\leq i\leq N}{r_i}\right)^M,\]
 which proves (i). \end{proof}

We now give a list conditions on an IFS arc $(\Ss,\gamma)$ in $\R^2$ that are easy to check and imply that the hypotheses of Theorem \ref{thm_BT_intro} are satisfied. Let $A_j \colon \reals^2 \to \reals^2$ be the orthogonal transformation associated to $S_j \in \Ss$. Then $A_j:\R^2\to\R^2$ is given by the composition of a counterclockwise rotation $R_{\alpha_j}$ about an angle $\alpha_j\in[0,2\pi)$ with a map $I_j$ which can be either the identity on $\R^2$ or a reflection through the $x$-axis. Thus, each similarity $S_j:\R^2\to\R^2$ can be written as  $S_j(x)=r_j (R_{\alpha_j}\circ I_j)(x)+b_j$, $x\in\R^2$, for some angle $\alpha_j\in[0,2\pi)$. Equivalently, in complex coordinates, $S_j(z)=r_je^{i\alpha_j}I_j(z)+b_j$, where  $I_j$ denotes either the identity on $\C$ or the complex conjugation $z\mapsto\bar{z}$.

\begin{cor}\label{cor_BT_intro_1}
Let $(\Ss,\gamma)$ be a normalized IFS arc in $\R^2$ with the property that 
\begin{enumerate}\setcounter{enumi}{-1}\item there exist numbers $t,s\in\N$ such that  $r_1^t=r_N^s$.
\end{enumerate} 
Assume that in addition one of the following conditions holds.
\begin{enumerate}\setcounter{enumi}{0}
\item  $S_1$ and $S_N$ are orientation preserving and $\alpha_1,\alpha_N \in 2\pi\Q$.
\item $S_1$ and $S_N$ are orientation reversing.
\item $S_1$ is orientation preserving, $S_N$ is orientation reversing, and $\alpha_1\in2\pi\Q$.
\item $\alpha_1=\alpha_N$.
\end{enumerate}
Then $(\Ss, \gamma)$ is a quasiarc. 
\end{cor}

\begin{proof}
Clearly, condition (0) in Corollary~\ref{cor_BT_intro_1} equals condition~(A) in Theorem~\ref{thm_BT_intro}. So it suffices to show that each of the conditions (1), (2), (3) and (4) of corollary~\ref{cor_BT_intro_1} implies condition~(B) of Theorem~\ref{thm_BT_intro}. 
Assume that condition~(1) holds. Then $A_1$ is the counterclockwise rotation $R_{\alpha_1}$ by the angle $\alpha_1$, and analogously $A_N=R_{\alpha_N}$. Moreover, there exists a number $q\in\N$ such that $q\alpha_1=q\alpha_N=0 \mod 2\pi$. Thus $A_1^q=R_{q\alpha_1}=\id_{\R^2}$ and $A_N^p=R_{p\alpha_N}=\id_{\R^2}$, and hence (B) is satisfied.
Now assume that condition~(2) holds. Then each of $A_1$ and $A_N$ is a reflection of $\R^2$ over some $1$-dimensional subspace of $\R^2$. Thus $A_1^2=A_N^2=\id_{\R^2}$, hence (B) is satisfied.  
A combination of the arguments for (1) and (2) shows that (3) implies (B) as well. Finally, assume that condition~(4) holds. Then $A_1^2=A_N^2$ (in case both $A_1$ and $A_N$ are orientation preserving, even $A_1=A_N$) and thus (B) is satisfied for $p=q=2$ (resp. $p=q=1$). \end{proof}

It is also not difficult to come up with special circumstances under which conditions similar to the ones postulated in Corollary~\ref{cor_BT_intro_1} work in $\reals^n$.
Consider for example the case when the similarities $S_1$ and $S_N$  of an IFS $\Ss$ in $\R^n$ are given by simple rotations: an orthogonal transformation $T$ of $\R^n$ is called a \emph{simple rotation} if there exists a two-dimensional subspace $A\subset\R^n$ such that $T$ restricted to $A$ is a rotation and $T$ restricted to the orthogonal complement of $A$ is the identity. 

Let us denote by $T_{A,\alpha}$ the simple rotation of $\R^n$ that is a counterclockwise rotation about the angle $\alpha\in[0,2\pi)$ on the (oriented) two-dimensional subspace $A\subset\R^n$ and leaves the orthogonal complement of $A$ fixed. 
Assume that for an IFS $\Ss$ in $\R^n$ there exists a two-dimensional subspace $A\subset\R^n$ and angles $\alpha_1,\alpha_N\in[0,2\pi)$ such that
$S_1(x)=r_1T_{A,\alpha_1}x$ (recall that $b_1=0$ since $(\Ss,\gamma)$ is a normalized IFS path) and $S_N(x)=r_1T_{A,\alpha_N}x+b_N$, for all $x\in\R^n$. Then, analogously to Corollary~\ref{cor_BT_intro_1}, the conditions 
\begin{enumerate}
\item There exist numbers $p,s\in\N$ and $r\in(0,1)$ such that $r_1=r^s$ and $r_N=r^p$
\item Either $\alpha_1,\alpha_N \in 2\pi\Q$, or $\alpha_1=\alpha_N$
\end{enumerate}
imply that $\gamma$ is a quasiarc.

\subsection{The cone containment condition}\label{sec_cone}
We now briefly describe another, more intuitive condition guaranteeing that an IFS arc $(\Ss, \gamma)$ is an IFS quasiarc. While it is quite broad, it requires knowledge of the invariant set $\gamma$ that cannot be easily checked from the similarities $\Ss$ alone. 

A \emph{closed cone} in $\reals^n$ is an isometric image of the set
$\{v \in \reals^n : \lan v,\bm{e}_1 \ran \geq \alpha ||v||\}$
for some $0 \leq \alpha < 1.$ Its \emph{apex} is the image of the vector $\bm{0}$. 

\begin{defin}\label{conecondition}
We say that an IFS arc $\gamma$ given by an IFS $\Ss$ satisfies  \emph{the cone containment condition} if for $i \in\{2,...,N-1\}$ there exist closed cones $L_i$ and $R_i$ with apex $S_i(\bm{e}_1)$ such that $L_i \cap R_i=\{S_i(\bm{e}_1)\}$, $S_i(\gamma)\subseteq L_i$ and $S_{i+1}(\gamma)\subseteq R_i$. 
\end{defin}
The following result was shown in \cite{Wen}: 
\begin{thm}[Wen and Xi, 2003]\label{thm_cone_cond} An IFS arc satisfying the cone containment condition has bounded turning.
\end{thm}
The IFS arcs (a.), (b.), and (c.)\ of Figure~\ref{pic_4frac} satisfy the cone containment condition. They also satisfy the conditions of Theorem~\ref{thm_BT_intro}. However, the IFS arc (d.)\ of Figure~\ref{pic_4frac} satisfies the conditions of Theorem~\ref{thm_BT_intro} but not the cone containment condition due to its ``spiraling" behavior:
\begin{example}\label{example_rotcurve}
Define an IFS $\Ss$ (using complex notation) by
\begin{equation*}\begin{array}{lllllll}
S_1(z)&=&\frac{1}{3}e^{i \arccos(\frac{3}{4})}z,& &
S_2(z)&=&\frac{1}{3}e^{-i \arccos(\frac{3}{4})}z+\frac{1}{4}+i\frac{\sqrt{7}}{12,}\\
S_3(z)&=&\frac{1}{3}e^{-i \arccos(\frac{3}{4})}z+\frac{1}{2},& &
S_4(z)&=&\frac{1}{3}e^{i \arccos(\frac{3}{4})}z+\frac{3}{4}-i\frac{\sqrt{7}}{12}.
\end{array}
\end{equation*}
Note that $e^{i \arccos(\frac{3}{4})}=\tfrac{3}{4}+\tfrac{\sqrt{7}}{4i} $. The following scheme illustrates the mappings $S_1,S_2,S_3,S_4$ applied to the line segment $I$:
\begin{center}
\includegraphics{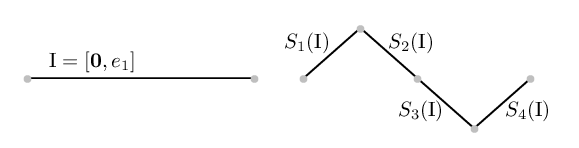}
\label{pic_IFS_rot}
\end{center}
Figure~\ref{pic_4frac}, image (d.) illustrates the invariant set $\gamma$ of $\Ss$. Since $\alpha_1$ and $\alpha_n$ are non-zero and the similarities in $\Ss$ are orientation preserving, the arc ``rotates" around each vertex. This makes it impossible for the cone containment condition to hold. However, once the non-trivial fact that $(\Ss,\gamma)$ is an IFS arc is established, it follows easily that it satisfies the conditions Theorem~\ref{thm_BT_intro}. Moreover, note that the similarity dimension of this IFS arc is $\frac{\log4}{\log3}$, which is the same as the similarity dimension of the other IFS arcs depicted in Figure~\ref{pic_4frac}. Thus by Corollary~\ref{bi-Lip equiv} it is bi-Lipschitz to each of the IFS paths. 
\end{example}

\begin{rmk}\label{rmk_sierpin} The cone containment condition does not imply that the invariant set $\gamma$ is a topological arc. To see this, consider the Sierpi\' nski gasket in $\R^2$ given by the IFS $\{S_1,S_2,S_3\}$ defined using complex notation by 
$ S_1(z)=\frac{1}{2} e^{i\pi/3}\ovl{z}, \ \ S_2(z)=\frac{1}{2}z+\left(\frac{1}{4}+\frac{\sqrt{3}}{4}i\right), \ \
S_3(z)=\frac{1}{2} e^{i\frac{5\pi}{3}}\bar{z}+\left(\frac{3}{4}+\frac{\sqrt{3}}{4}i\right).$ Figure~\ref{pic_gasket} illustrates the mappings $S_1,S_2,S_3$. Note that this example also shows that the invariant set $\gamma$ can fail to be an arc even though $T^k(I)$ is an arc for each $k \in \nats$; recall that $T$ is the Hutchinson operator associated to $\Ss$ and $I$ is the interval connecting $\0$ and $\bm{e}_1$ in $\reals^2$.
\begin{figure}[H]
\includegraphics{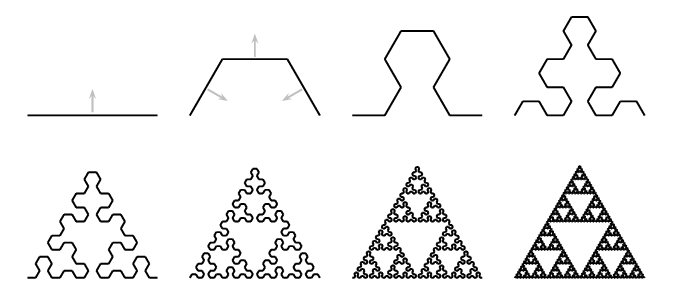}
\caption{Sets $I$ and $T^k(I)$ for $k=1,...,5$. The grey arrows illustrate the orientation of the mappings $S_1, \ S_2$ and $S_3$.}\label{pic_gasket}
\end{figure}
\end{rmk}

\section{Parameterizations of IFS paths}\label{sec_param}
In this section we construct optimally H\"older continuous parameterizations of IFS paths, and show that these parameterizations are actually bi-H\"older when the IFS path is an 	IFS quasiarc. 
\subsection{The construction of structural parameterizations}\label{sec_param_construct}

In \cite{hutchinson1981} Hutchinson gives a remark explaining how to construct a parameterization for an IFS path, as follows. Let $(\Ss,\gamma)$ be an IFS path. Choose a partition $\mathcal{P}$ of $[0,1]$ into $N$ pieces, i.e. fix points $0=t_{0}<t_{1}<t_{2}<...<t_{N}=1$.  We define a new IFS $\s=\{s_{1},...,s_{N}\}$ in $\R$ based on this partition and the original IFS $\Ss$: For $i=1,...,N$, define the mapping $s_{i}:\R\rightarrow \R$ by
$ s_i(t):= t\, t_{i} +(1-t) t_{i-1} $. Note that for any $i\in\{1,...,N\}$, the similarity $s_i$ maps the interval $[0,1]$ onto the interval $[t_{i-1},t_{i}]$, and that the ratio of $s_i$ is $(t_i-t_{i-1})$. 
The invariant set of the IFS $\s$ is the interval $[0,1]$. Moreover, $(\s,[0,1])$ is an IFS arc in $\R$. In particular, for any $k\in\N$, we can write  \[[0,1]=\bigcup_{\sigma\in\{1,...,N\}^k}s_\sigma([0,1]).\]
This union is disjoint except for the vertices of generation $k$, where adjacent copies of $[0,1]$ overlap in the way described in Definition~\ref{def_inj}. Let $\phi_0\colon [0,1] \to \R^n$ be the inclusion defined by $\phi_0(t):=t\bm{e}_1$ for all $t\in[0,1]$. Define $\phi_k: [0,1]\rightarrow \R^n$ by $\phi_k(t)=S_\sigma\circ \phi_{0}\circ s_\sigma^{-1}(t)$ for $t\in\s_\sigma([0,1])$  with $\sigma \in \{1,\hdots, N\}^k$. It follows that $\phi_k([0,1])=T^k(I)$, where $T$ is the Hutchinson operator associated to $\Ss$ and $I$ is the line segment connecting $\bm{0}$ to $\bm{e}_1$. Hutchinson showed in \cite{hutchinson1981} that the maps $\phi_k$ converge uniformly to a continuous map $\phi:[0,1]\rightarrow \R^n$ with $\phi([0,1])=\gamma$. 

Let $k \in \nats$ be an integer and let $\sigma \in \{1,\hdots,N\}^k$. If $t \in s_i \circ s_{\sigma}([0,1])$, then 
$\phi_{k+1}(t) = S_i \circ \phi_k \circ s_i\inv(t)$, and hence
it follows that 
\begin{equation}\label{lem_param}
S_\sigma \circ \phi=\phi \circ s_\sigma,
\end{equation}
i.e., $\phi$ is a \emph{structural parameterization} of $(\Ss,\gamma)$. 

It turns out that there is a canonical choice of the partition $\mathcal{P}$ that optimizes the metric distortion of $\phi$:

\begin{defin} Let $(\Ss,\gamma)$ be an IFS path with similarity dimension $s$. The \emph{Hutchinson parameterization of $\gamma$} is the limit $\phi$ of the sequence of maps $\{\phi_k\}_{k \in \nats}$ defined above where the partition $\mathcal{P}$ satisfies  $t_i-t_{i-1}=r_i^s$ for $i\in\{1,...,N\}$.
\end{defin}

We now justify the nomenclature ``IFS arc". 
\begin{prop}\label{prop_arc_iff_inj}
Let $(\Ss,\gamma)$ be an IFS path. Then the following are equivalent:
\begin{enumerate}
\item $(\Ss,\gamma)$ is an IFS arc, 
\item the Hutchinson parameterization $\phi$ of $\gamma$ is injective.
\end{enumerate}
\end{prop}

\begin{proof}[Proof of Proposition~\ref{prop_arc_iff_inj}]
Let $\gamma$ be an IFS arc. Assume, towards a contradiction, that the Hutchinson parameterization $\phi$ of $\gamma$ is not injective. Thus there exist $u,v\in[0,1]$, $u< v$, such that $\phi(u)=\phi(v)$. As in Section~\ref{sec_estimates}, let $m$ be the earliest generation separating $u,v$ with respect to $(\s, [0,1])$, and $q$ a vertex of generation $m$ such that $u<q<v$. Thus there exist $\sigma\in\{1,...,N\}^{m-1}$ and $i<j \in\{1,...,N\}$ such that $u\in s_\sigma s_i([0,1]) \ \text{and}\ v\in s_\sigma s_j([0,1])$  with 
$u\neq s_\sigma s_i(1)\ \text{and}\ u\neq s_\sigma s_j(0).$

Then, since $\phi(u)=\phi(v)$, the set $S_\sigma S_i(\gamma)\cap S_\sigma S_j(\gamma)$ contains a point which is not $S_\sigma S_i(\bm{e}_1)$ or $S_\sigma S_j(\0)$. Applying $S_\sigma^{-1}$ yields a contradiction.

That $(\Ss, \gamma)$ is an IFS arc if $\phi$ is injective follows quickly from the definitions.
\end{proof}

\subsection{H\"older continuity of the Hutchinson parameterization}
The following result can be found in \cite{hata1985} and \cite{Aseev}, but we include the proof for completeness as it is a simple consequence of the basic estimates of Section~\ref{sec_estimates}.

\begin{thm}\label{thm_hoelder}
Let  $(\Ss,\gamma)$ be an IFS path with similarity dimension $s$. Then Hutchinson parameterization of $\gamma$ is $\frac{1}{s}$-H\"older.
\end{thm}
 
\begin{proof}Let $u<v$ be distinct points of $[0,1]$. We will employ the estimates of Section~\ref{sec_estimates} with respect to the IFS path $\mathbf{s}$ defined above. Namely,
\begin{itemize}
\item[(1)] let $m$ be the smallest generation separating $u$ and $v$,
\item[(2)] let $q$ be a generation $m$ vertex satisfying $u \leq q \leq v$,
\item[(3)] let $\sigma \in \{1,\hdots,N\}^{m-1}$ be the sequence satisfying $[u,v] \subeq s_\sigma([0,1])$,
\end{itemize}
If $(u,v)$ is a case-1-pair, then (3) above, \eqref{lem_param}, and the fundamental estimate \eqref{case1_1} imply that 
$$\dm(\phi(u),\phi(v)) \leq \diam(\gamma)r_{\sigma}\leq \frac{\diam(\gamma)}{D_{\mathbf{s}}} |u-v|^{\frac{1}{s}}.$$ 
Now, assume that $(u,v)$ is a case-2-pair and thus $(u,q,v)$ is a case-2-triple. We may assume without loss of generality 
\begin{equation}\label{another_estimate}
\dm(\phi(u),\phi(v))\leq 2 \dm(\phi(u),\phi(q)).
\end{equation}As in Section~\ref{sec_estimates},
\begin{itemize}
\item let $\sigma_m \in \{1,\hdots,N-1\}$ be the number such that 
$ s_{\sigma\sigma_m}(1)=z=s_{\sigma(\sigma_m+1)}(0),$
\item let $k \in \nats$ be the number satisfying 
$u \in s_{\sigma\sigma_m N^{k}}([0,1]) \ \text{and}\ u\notin  s_{\sigma\sigma_m N^{k+1}}([0,1])$.
\end{itemize}
Then $u,q\in s_{\sigma,\sigma_m,N^k}([0,1])$, and so $\phi(u),\phi(q)\in S_{\sigma,\sigma_m,N^k}(\gamma)$. Thus, by applying \eqref{another_estimate},
$$\dm(\phi(u),\phi(v))\leq 2\diam(\gamma)r_\sigma r_{\sigma_m}r_N^k.$$ On the other hand, using \eqref{case2_1_x} and the fact that $u \leq q \leq v$ are points in $[0,1]$, we see that $$|u-v|\geq |u-q| \geq D_{\mbf{s}}(r_\sigma r_{\sigma_m}r_N^k)^s,$$
and thus $\dm(\phi(u),\phi(v))\leq 2 \frac{\diam(\gamma)}{D_{\mbf{s}}}|u-v|^{\frac{1}{s}}$. This completes the proof. 
\end{proof}

\subsection{Bi-H\"older continuity of the Hutchinson parameterization}\label{sec_bi_hoelder}

We now use the estimates of Section~\ref{sec_estimates} to prove Theorem~\ref{thm_bi}. 
 
\begin{proof}[Proof of Theorem \ref{thm_bi}] Let $(\Ss,\gamma)$ be an IFS quasiarc and let $s$ denote its similarity dimension.  Let $u< v$ be distinct points of $[0,1]$. Given Theorem \ref{thm_hoelder}, in order to establish the $\frac{1}{s}$-bi-H\"older continuity of the Hutchinson parameterization $\phi$, we must only produce an appropriate lower bound for $\dm(\phi(u),\phi(v))$. 

Let $(\s, [0,1])$ be the IFS arc used in the construction of $\phi$ (see Section~\ref{sec_param_construct}), and let $m$, $q$, and $\sigma$ be as defined in (1)-(3) of the proof of Theorem \ref{thm_hoelder}. First, assume that $(u,v)$ is a case-1-pair with respect to the IFS arc $(\s,[0,1])$. Then $(\phi(u),\phi(v))$ is a case-1-pair with respect to the IFS arc $(\Ss, \gamma)$ and so by \eqref{case1_1} applied to $(\Ss, \gamma)$ and \eqref{case1_2} applied to $(\s, [0,1])$, 
$$\dm(\phi(u),\phi(v)) \geq D_{\Ss} |u-v|^{\frac{1}{s}}.$$

Now assume that $(u,q,v)$ is a case-2-triple with respect to the $(\s,[0,1])$. Then $(\phi(u),\phi(q),\phi(v))$ is a case-2-triple with respect to $(\Ss,\gamma)$. Without loss of generality we may assume that $|u-q| \geq |q-v|$. Let $C$ denote the bounded turning constant of $\gamma$. Then by Lemma~\ref{lem_3pt_BT}, 
$$\dm(\phi(u),\phi(v))\geq \frac{\dm(\phi(u),\phi(q))}{C}.$$ 
Let $\sigma_m \in \{1,\hdots, N-1\}$ and $k \in \nats$ be so that \eqref{choose k} holds for the triple $(\phi(u),\phi(q),\phi(v))$.  Then \eqref{case2_1_x} for this triple implies 
$$\dm(\phi(u),\phi(q))\geq D_{\Ss} r_\sigma r_{\sigma_m}r_N^k.$$ 
By \eqref{lem_param}, $\sigma_m$ and $k$ would remain unchanged had we chosen them based on the triple $(u,q,v)$ with respect to the IFS $(\s,[0,1])$, and so \eqref{case2_2_x} applied to the IFS $\s$ yields
$$|u-q| \leq (r_\sigma r_{\sigma_m} r_N^k)^s.$$
As we have assumed $|u-q| \geq |q-v|$, we conclude that 
$$\dm(\phi(u),\phi(v))\geq  \frac{D_{\Ss}}{2C}|u-v|^\frac{1}{s}.$$

Conversely, if the Hutchinson parameterization $\phi$ is $\frac{1}{s}$-bi-H\"older continuous, then it is also quasisymmetric. The bounded turning condition is preserved under quasisymmetries \cite{TV}, and so the fact that $[0,1]$ has bounded turning implies that $\gamma$ has bounded turning. Thus $\gamma$ is a quasiarc. 
\end{proof}

In fact, if $(\Ss,\gamma)$ is an IFS quasiarc, then the Hutchinson parameterization is the $s$-dimensional arclength parameterization of $\gamma.$

\begin{cor}\label{cor_arclength} Let $(\Ss, \gamma)$ be an IFS quasiarc with similarity dimension $s$. Then the Hutchinson parameterization $\phi$ satisfies
$$\Hdim^s(\phi[u,v]) =\Hdim^s(\gamma)|u-v|,$$ for all $u,v\in[0,1]$.
 \end{cor}
 
 \begin{proof} 
For the moment, fix $\sigma\in\{1,...,N\}^k$ and $k\in\N$. Since $\phi$ is $\frac{1}{s}$-bi-H\"older continuous, it holds that $0<\Hhh^s(\gamma)<\infty$. Hence,
$\Hhh^s(S_\sigma(\gamma))=\Hhh^s(\gamma) r_\sigma^s.$
On the other hand, $r_\sigma ^s=\Hhh^1(s_\sigma([0,1]))$. Thus, since $\phi(s_\sigma([0,1]))=S_\sigma(\gamma)$, it follows that
\begin{equation}\label{eq_arc_lenght_2}
\Hhh^s(\phi(s_\sigma([0,1])))=\Hhh^s(\gamma)\Hhh^1(s_\sigma([0,1])).
\end{equation} 
Let $u,v\in[0,1]$. Define the sequence of collections of intervals $(M_k)_{k\in\N}$ as follows: 
\begin{equation*}\begin{split}
M_1 &=\{s_\sigma([0,1])\subseteq [u,v]\colon \sigma\in\{1,...,N\}\}\\
M_k &=\{s_\sigma([0,1])\subseteq [u,v]\colon \sigma\in\{1,...,N\}^k, s_\sigma([0,1])\nsubseteq I \text{ for any }I\in M_j \text{ with } j<k  \}\\
\end{split}\end{equation*} Then 
$$ (u,v)\subeq \bigcup_{k\in\N}\bigcup_{I\in M_k} I \subeq [u,v].$$ Moreover, any two intervals in this union intersect in at most one point. Hence,
$$ \phi((u,v))\subeq\bigcup_{k\in\N}\bigcup_{I\in M_k} \phi(I) \subeq \phi([u,v]).$$ 
Since $(\Ss,\gamma)$ is an IFS arc, Proposition~\ref{prop_arc_iff_inj} implies that any two of the terms above intersect in at most one point as well.  Since points have zero $s$-dimensional Hausdorff measure, \eqref{eq_arc_lenght_2} implies that 
\[ \Hhh^s(\phi[u,v]) = \sum_{k\in\N}\sum_{I\in M_k} \Hhh^s(\phi(I))= 
 \sum_{k\in\N}\sum_{I\in M_k} \Hhh^s(\gamma)\Hhh^1(I)=\Hhh^s(\gamma)\Hhh^1([u,v]),\]
 as desired.
\end{proof}

To put Corollary~\ref{cor_arclength} in context, it can be gleaned from \cite{mclaughlin1987}, \cite{falconer1989}, and \cite{falconermarsh1989}, or slightly more directly from \cite{ghamsari-herron1998}, that any Ahlfors $s$-regular quasiarc has an $s$-dimensional arclength parameterization by $[0,1]$ that is  $\frac{1}{s}$-bi-H\"older continuous; the novelty of Corollary~\ref{cor_arclength} is that for IFS quasiarcs, this parameterization coincides with the Hutchinson parameterization. 

\subsection{An Example}\label{sec_ex}
Define the IFS $\Ss(a,b)$ depending on two parameters $a,b\in(0,\frac{1}{2})$ (using complex notation) by 
\begin{equation*}\begin{array}{lllllll}
S_1(z)&=& az, & &
S_2(z)&=&be^{i \arccos(b^{-1}(\frac{1}{2}-a))}z+a, \\
 S_3(z)&=&be^{-i \arccos(b^{-1}(\frac{1}{2}-a))}z +\frac{1}{2}+ih,& &
 S_4(z)&=& az+(1-a),
\end{array}
\end{equation*} where $h=\sqrt{b^2-(\frac{1}{2}-a)^2}$. 

The following figure illustrates the mappings $S_1,S_2,S_3,S_4$ applied to the segment $I$.
 \begin{center}
 \includegraphics{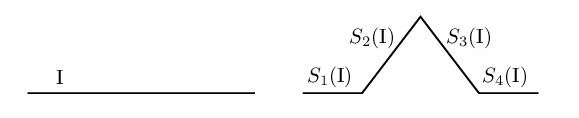}
 \end{center}
One can easily show that for any parameters $a,b\in(0,\frac{1}{2})$, the invariant set $\gamma(a,b)$ of $\Ss(a,b)$ is a quasiarc (either by applying Theorem~\ref{thm_BT_intro}, or using the cone containment condition). The similarity dimension (and hence Hausdorff dimension) of $\gamma$ is the unique solution $s$ of $2a^s+2b^s=1$.  Thus, for each $s \in (0,2)$, such examples provide a one-parameter family of bi-Lipshitz equivalent  IFS quasiarcs of dimension $s$. Figure \ref{pic_4frac}, images (a.), (b.) and (c.), are members of this family where $s=\frac{\log(4)}{\log(3)}$; image (d.) is also an IFS quasiarc of this dimension and hence bi-Lipschitz equivalent to the other arcs displayed, but arises from the IFS described in Example \ref{example_rotcurve}. 

\subsection{Approximate self-similarity and bounded turning}\label{sec_FM} 
In many contexts, the space $(X,d)$ of interest is not known to be the invariant set of a family of contracting similarities, but rather only known to posses a more flexible form of self-similarity: 
\begin{itemize} 
\item scaled bi-Lipschitz copies of any small-scale piece of $(X,d)$ appear at the top scale, and
\item a scaled bi-Lipschitz copy of $(X,d)$ appears at every scale and location.
\end{itemize}
A metric space satisfying the first condition above was called \emph{quasi-self-similar} by Mclauglin \cite{mclaughlin1987} and \emph{approximate self-similar} in \cite{Bourdon}; a metric space homeomorphic to a circle that satisfies both conditions was called a \emph{quasicircle} by Falconer-Marsh \cite{falconermarsh1989}. Carrasco Piaggio has shown that an approximately self-similar locally connected metric space has bounded turning \cite{Piaggio}, which implies  that a ``Falconer-Marsh quasicircle" is indeed the quasisymmetric image of $\mathbb{S}^1$ (the converse of this statement is false). 

Taken together, the papers of McLaughlin \cite{mclaughlin1987}, Falconer \cite{falconer1989}, and Falconer-Marsh \cite{falconermarsh1989} imply the following theorem:

\begin{thm}\label{thm_MFM} A Falconer-Marsh quasicircle possess a bi-H\"older parameterization by $\mathbb{S}^1$.
\end{thm}

Thus, an alternate approach to Theorem \ref{thm_bi} is to show that the invariant set of an IFS quasiarc satisfies the two conditions given above and then prove a version of Theorem \ref{thm_MFM} for arcs. This is much less direct than the approach we have taken; however, it turns out to be logically equivalent, as the following result shows:

\begin{thm}\label{lem_blow-up-down}
The invariant set of an IFS arc $(\Ss, \gamma)$ has bounded turning if and only if the following two conditions hold:
\begin{enumerate}
\item There exist constants $r_0>0$ and $C_1, C_2 > 0$ such that for every open set $U\subseteq \R^n$ satisfying $U \cap \gamma \neq \emptyset$ and $0< \diam (U \cap \gamma)  < r_0$, there exists a mapping $\varphi : U\cap \gamma \rightarrow \gamma$ such that
\[\frac{C_1\dm(a,b)}{\diam(U \cap \gamma)} \leq \dm ( \varphi (a), \varphi (b)) \leq  \frac{C_2\dm(a,b)}{\diam(U \cap \gamma)}\] 
for all $a,b\in U\cap \gamma$
\item There exist constants $r_1>0$ and $C_3 > 0$ such that for any $x \in \gamma$ and radius $0<r < r_1$, there is a mapping $\psi : \Gamma \rightarrow B(x,r) \cap \Gamma$ satisfying
\[\simd( \psi(a),\psi(b))\geq C_3 r \simd (a,b) \] for all $a,b\in \gamma$.
\end{enumerate}
\end{thm}
\begin{proof} That bounded turning implies the two conditions given in the statement follows immediately from Theorem \ref{thm_bi}, as they are invariant under bi-H\"older changes of metric and are clearly satisfied by the interval $[0,1]$. 

Now suppose that $(\Ss,\gamma)$ is an IFS arc such that condition (1) of Theorem \ref{lem_blow-up-down} holds. That $\gamma$ has bounded turning follows from the work of Piaggio Carrasco \cite{Piaggio}, but we can give a short, self-contained proof in this special setting. 

Assume, towards a contradiction, that $\gamma$ is not of bounded turning. Then there is a sequence of pairs of points $\{(x_i,y_i)\}_{i \in \nats}$ of $\gamma$ such that $x_i < y_i$, and 
$$\lim_{i \to \infty} \frac{\dm(x_i,y_i)}{\diam (\gamma_{x_i,y_i})} = 0.$$
Define $\rho_i := \diam(\gamma_{x_i,y_i}).$
As $\gamma$ is an arc, it must be the case that
$\lim_{i \to \infty}\rho_i=0$. Hence, we may assume without loss of generality that $2\rho_i  < r_0$ for each $i \in \nats$. 

Fix $i \in \nats$. Let $\varphi_i \colon B(x_i,2\rho_i) \cap \gamma \to \gamma$ be the mapping provided by condition (1) with $U=B(x_i,2\rho_i)$. 
Since the mapping $\varphi_i$ is a topological embedding, $\varphi_i(\gamma_{x_i,y_i})=\gamma_{\varphi_i(x_i),\varphi(y_i)}.$ Hence, the estimates of condition (1) yield
\begin{equation}\label{lowbd}\diam(\gamma_{\varphi_i(x_i),\varphi_i(y_i)}) \geq \frac{C_1\diam(\gamma_{x_i,y_i})}{\diam(U\cap\gamma)} 
\end{equation}
as well as
\begin{equation}\label{same_ratio}\lim_{i \to \infty} \frac{\dm(\varphi_i(x_i),\varphi_i(y_i))}{\diam(\gamma_{\varphi_i(x_i),\varphi_i(y_i)})} \leq\lim_{i \to \infty} \frac{C_2\dm(x_i,y_i)}{C_1\diam(\gamma_{x_i,y_i})}=0.\end{equation}

However, $\diam(U\cap\gamma)\leq 2\rho_i =2 \diam (\gamma_{x_i,y_i})$ and thus \eqref{lowbd} implies that $\diam(\gamma_{\varphi_i(x_i),\varphi_i(y_i)}) \geq  \frac{C_1}{2}.$
Since $\gamma$ is an arc, this contradicts \eqref{same_ratio}. 
\end{proof}
\begin{rmk} As the proof shows, condition (1) alone implies that an IFS arc is an IFS quasiarc.
\end{rmk} 

\section{IFS paths with non-injective Hutchinson parametrization}\label{sec_non_inj}

The goal of this section is to prove Theorem~\ref{thm_non-inj} and to give several sets of sufficient conditions in order for Theorem~\ref{thm_non-inj} to hold.

\begin{proof}[Proof of Theorem~\ref{thm_non-inj}]
Recall that $z:=S_i(\bm{e}_1)=S_{i+1}(\0)$. By condition (3.b) of the statement of Theorem~\ref{thm_non-inj}, the set
$$T := \{\theta \in [0,2\pi): [S_i(\gamma)\cap R_\theta(S_{i+1}(\gamma))]\backslash\{z\} \neq \emptyset\}$$ contains an open interval $I$.

Since $t\alpha_1-s\alpha_N \in 2\pi(\R\backslash \Q)$, there exists $m\in \N$ such that $\theta_0:=m(t\alpha_1-s\alpha_N)\in I\subseteq T$. Thus the sets $S_i(\gamma)$ and   $R_{\theta_0}(S_{i+1}(\gamma))$ intersect in a point $\tilde{z}\neq z$. In particular, this implies that $S_i(\gamma)\cup R_{\theta_0}(S_{i+1}(\gamma))$ contains a homeomorphic copy of $\sphere^1$.

Regarding condition (3.a), assume that both $S_i$ and $S_{i+1}$ are orientation preserving; the case that they are both orientation reversing is analogous. Now consider the sets $S_i\circ S_N^{ms} (\gamma), \ S_{i+1}\circ S_1^{mt}(\gamma)\subset \gamma$. 
By condition (1), the similarity $S_N$ rotates $\gamma$ by an angle of $\alpha_N$ \emph{around the point $\bm{e}_1$}, while $S_1$ rotates $\gamma$ by an angle of $\alpha_1$ around $\0$.  

Therefore the angle between $S_i\circ S_N^{ms} (\gamma)$ and $\ S_{i+1}\circ S_1^{mt}(\gamma)$ at $z$ is the angle between $S_i(\gamma)$ and $S_{i+1}(\gamma)$ at $z$ plus $m( t\alpha_1-s\alpha_n)$.
By condition (2), $r_1^{mt}=r_N^{ms}$, and hence $(S_i\circ S_N^{ms} (\gamma))\cup (\ S_{i+1}\circ S_1^{mt}(\gamma))$ is the image of $S_i(\gamma)\cup R_{\theta_0}(S_{i+1}(\gamma))$ under a similarity. Thus, $\gamma$ contains a homeomorphic copy of $\sphere^1$, and hence $\gamma$ can not be the homeomorphic image of an interval.
\end{proof}

We now indicate one situation in which the hypotheses of Theorem~\ref{thm_non-inj} hold. While somewhat simpler, it still requires some \emph{a priori} knowledge of the IFS path. 

\begin{cor}\label{cor_non-inj_new}
Let $(\Ss,\gamma)$ be a normalized IFS path in $\R^2$ that satisfies conditions (1) and (2) of Theorem~\ref{thm_non-inj}. Assume that 
\begin{enumerate}[(3*)]
\item\begin{enumerate}[(a)]
\item either both $S_i$ and $S_{i+1}$ are orientation preserving, or both are orientation reversing,
\item there is no rotation around the point $S_i(\bm{e}_1)=S_{i+1}(\0)$ that maps $S_i(\gamma)$ into $S_{i+1}(\gamma)$, or vice versa.
\end{enumerate}
\end{enumerate}
Then condition (3) of Theorem~\ref{thm_non-inj} holds and so $\gamma$ is not an IFS arc. 
\end{cor}

In order to prove this, we will need a topological lemma. Consider the punctured disk 
 $$ A := \left\{z \in \comps: 0< |z| < 1\right\},$$
 and its closure $\ovl{A}$, which coincides with the closed unit disk. 
For each $\theta \in \mathbb{S}^1 \subeq \comps,$ define a mapping $R_\theta \colon \comps \to \comps$ by 
$R_\theta(z) = \theta z.$  

\begin{lem}\label{lem_rotation} Let  $\alpha \colon [0,1] \to \ovl{A}$ be an injective path. Consider another injective path $\beta \colon [0,1] \to \ovl{A}$  such that  
\begin{itemize} 
\item $\beta(0)=0$,
\item $\beta(s) \in A$ for all $s \in (0,1)$,
\item $|\beta(1)|=1$,
\item $\im \alpha \cap \im \beta \subeq \{0\}$ (where $\im \alpha$ and $\im \beta$ denote the images of $\alpha$ and $\beta$, respectively).
\end{itemize}

Furthermore, assume that there does not exist $\theta \in \mathbb{S}^1$ such that 
$$ \im (R_\theta \circ \alpha) \subeq \im (\beta).$$

Then the set $$T := \{\theta \in \mathbb{S}^1: (\im (R_\theta \circ \alpha) \cap \im (\beta))\backslash \{0\} \neq \emptyset\}$$ 
contains an interval. 
\end{lem}

\begin{proof}[Proof of Lemma~\ref{lem_rotation}] Consider the strip $\til{A} =  (0, 1]\times \reals$ as the universal cover of the punctured closed disk $A'=\ovl{A} \bslash \{0\}$ with the associated projection $\pi \colon \til{A} \to A'$ defined by $\pi(r,t) = re^{it}$. Let $\til{\alpha}$ and $\til{\beta}$ be lifts of $\alpha|_{(0,1]}$ and $\beta|_{(0,1]}$ to $\til{A}$ respectively; then $\til{\alpha}$ and $\til{\beta}$ have disjoint images. 

For $t \in \reals$, define $\til{\alpha}_t$ by $\til{\alpha}+(0,t)$. Then the image of $\pi(\til{\alpha}_t)$ coincides with the image of $R_{\theta}\circ\alpha$ for $\theta=e^{it}.$  Thus, it suffices to show that the set
$$\til{T} = \{t \in \reals: (\im (\til{\alpha}_t) \cap \im(\til{\beta}))\neq \emptyset\}$$
contains an interval.  

By assumption, there is no $t \in \reals$ such that $\im (\til{\alpha}_t) \subeq \im(\til{\beta})$.  Note that $\til{A} \bslash \til{\beta}$ has exactly two components, which we denote by $L$ and $U$. Define
$$t_0 := \sup\{t \in \reals: \im \til{\alpha}_t \subeq  L\} \ \text{and}\ t_1 := \inf\{t \in \reals: \im \til{\alpha}_t \subeq  U\}.$$
We may assume without loss of generality that $0\leq t_0 < \infty$ and that $t_0 \leq  t_1 < \infty$. Suppose that $t_0 < t_1$; then for each $t \in (t_0, t_1)$, the set $\im(\til{\alpha}_t)$ is neither contained in $L$ nor $U$, and so $t \in \til{T}$. This shows that $\til{T}$ contains an interval. 

Now suppose that $t_0 = t_1$. Then, for any $\ep>0$, it holds that  $\im(\til{\alpha}_{t_0-\ep})\subeq L$ while $\im(\til{\alpha}_{t_0 + \ep}) \subeq U$. This implies that $\im(\til{\alpha}_{t_0})$ is contained in $\im(\til{\beta})$, a contradiction. 
\end{proof} 

\begin{proof}[Proof of Corollary~\ref{cor_non-inj_new}]
Let $(\Ss,\gamma)$ be a normalized IFS path in $\R^2$ that satisfies conditions $(1)$ and  $(2)$ of Theorem~\ref{thm_non-inj}.

By assumption, we may find an index $i \in \{1,\hdots, N-1\}$ so that both $S_i$ and $S_{i+1}$ are orientation preserving or both are orientation reversing, and moreover that there is no rotation around $z:=S_i(\bm{e}_1)=S_{i+1}(\0)$ that maps $S_i(\gamma)$ onto a subset of $S_{i+1}(\gamma)$ or vice-versa. We assume that 
$$\max\{\dm(x,z)\colon x\in S_i(\gamma)\} \leq R:=\max\{\dm(z,y)\colon y\in S_{i+1}(\gamma)\},$$
as a similar argument is valid if this is not the case.  Furthermore, define
$$t_0:=\psi^{-1}(S_{i}(\0)),\  t_1:=\psi^{-1}(z), \ \text{and}\  t_2:=\min\{t\in \psi^{-1}(S_{i+1}(\gamma))  \colon \dm(z,\psi(t))=R\}.$$ Consider the punctured disk $A:=\{p\in \R^2\colon 0<\dm(z,p)<R\}$ and the injective paths $\alpha, \beta$ defined by $\alpha:=\psi|_{[t_0,t_1]}$ and  $\beta=\psi|_{[t_1,t_2]}$. Note that thus $\beta(t_1)=z$, $\beta(s)\in A$ for all $s\in(t_1,t_2)$, $|\beta(t_2)|=R$ and $\im \alpha\cap \im \beta =\{z\}$. Moreover, as we have stated above, there does not exist an angle $\theta\in\sphere^1$ such that $\im (R_\theta\circ \alpha)\subseteq \im(\beta)$. Thus, after shifting, scaling, and reparameterizing, Lemma~\ref{lem_rotation} shows the set 
$$T := \{\theta \in \mathbb{S}^1: (\im (R_\theta \circ \alpha) \cap \im (\beta))\backslash\{z\} \neq \emptyset\}$$ contains an interval $I$.
\end{proof}

\begin{rmk} Theorem~\ref{thm_non-inj} still holds under either one of the following two sets of slightly different assumptions:
Let $(\Ss,\gamma)$ be a normalized IFS path in $\R^2$ such that
\begin{enumerate}[(1')]
\item $S_1$ is orientation preserving, $S_N$ is orientation reversing.
\item there exists a number $t\in\N$ such that $r_1^t=r_N^2$, and 
 $t\alpha_1  \in 2\pi(\R\backslash\Q),$
\item there exists $i\in\{1,...,N-1\}$ such that the following conditions hold: 
\begin{enumerate}[a)]
\item $S_i$ and $S_{i+1}$ are either both orientation preserving or both orientation reversing,
\item the set 
$\{\theta\in[0,2\pi):\ [S_i(\gamma)\cap R_\theta(S_{i+1}(\gamma))]\backslash\{z\} \neq \emptyset\}$
contains an open interval.
\end{enumerate}
\end{enumerate}
and 
\begin{enumerate}
\item[(1'')] $S_1$ and $S_N$ are orientation preserving,
\item[(2'')] there exist numbers $t,s\in\N$ such that $r_1^t=r_N^s$, and 
 $(t\alpha_1 +s\alpha_N) \in 2\pi(\R\backslash\Q),$
\item[(3'')] there exists $i\in\{1,...,N-1\}$ such that the following conditions hold: 
\begin{enumerate}[a)]
\item $S_i$ is orientation preserving and $S_{i+1}$ is orientation reversing (or vice versa),
\item the set 
$\{\theta\in[0,2\pi):\ [S_i(\gamma)\cap R_\theta(S_{i+1}(\gamma))]\backslash\{z\} \neq \emptyset\}$
contains an open interval.
\end{enumerate}
\end{enumerate}

The proofs are analogous to the proof of Theorem~\ref{thm_non-inj} by just carefully reconsidering with what angle $S_N$ and $S_1$ rotate $\gamma$ around $0$ and $\bm{e}_1$ respectively, and taking into account how an orientation reversing similarity transforms angles. Also one can deduce corollaries analogous to Corollary~\ref{cor_non-inj_new} from the above variants of Theorem~\ref{thm_non-inj}. \end{rmk}

\section{Open questions and further directions for research}
The obvious task left uncompleted by this work is an optimal version of Theorem \ref{thm_BT_intro}:
\begin{question} Give necessary and sufficient conditions for an IFS arc $(\Ss, \gamma)$ to be an IFS quasiarc in terms of the similarities in $\Ss$ alone.
\end{question} 
It would be particularly interesting if there was a simple characterization of IFS quasiarcs, rather than an exhaustive list of cases. A similar question can be posed regarding necessary and sufficient conditions for an IFS path to be an IFS arc; this question is likely very difficult. More approachable is the question of sharpness of Theorem \ref{thm_non-inj}. 

\begin{question} Does Theorem~\ref{thm_non-inj} hold without the assumption of condition (3)?
\end{question} 

The bounded turning condition makes sense for arbitrary metric spaces, and so one can inquire about whether or not the invariant set of an IFS path has this condition regardless of the topological type of its invariant. If the invariant set of an IFS path fails to be an arc, then philosophically it should be easier to verify the bounded turning condition.

\begin{question} Give necessary and sufficient conditions for an IFS path $(\Ss, \gamma)$ \emph{that is not an IFS arc} to have bounded turning in terms of the similarities in $\Ss$ alone.
\end{question}

Characterizing metric spaces that are quasisymmetrically equivalent to the standard two-sphere is an important problem in geometric group theory \cite{BK} and in the dynamics of rational mappings \cite{BonkMeyer}. In \cite{Meyer}, Meyer gave examples of metric spaces that are quasisymmetrically equivalent to the standard sphere but have non-integer Hausdorff dimension; these examples were called \emph{snowspheres} and can be considered as a two-dimensional version of Rohde's snowflakes. 

\begin{question} Can a large and concrete class of ``IFS snowspheres", including many with unequal scaling ratios, be shown to be quasisymmetrically equivalent to the standard sphere? For such snowspheres, does the Hausdorff dimension determine the invariant set up to bi-Lipschitz equivalence? 
\end{question}

The possible homeomorphism types for invariant sets of IFS paths include many specific fractals, such as the Sierpinski gasket, as shown in Section \ref{sec_cone}. It is also possible to realize the Sierpinski carpet as the invariant set of an IFS path (see Figure~\ref{pic_carpet}). The quasisymmetric geometry of the Sierpinski carpet also plays an important role in dynamics and geometric group theory \cite{BonkICM}. Perhaps viewing the carpet as an IFS path can yield some insight into this subject. 

\begin{center}
\begin{figure}[h]
\includegraphics{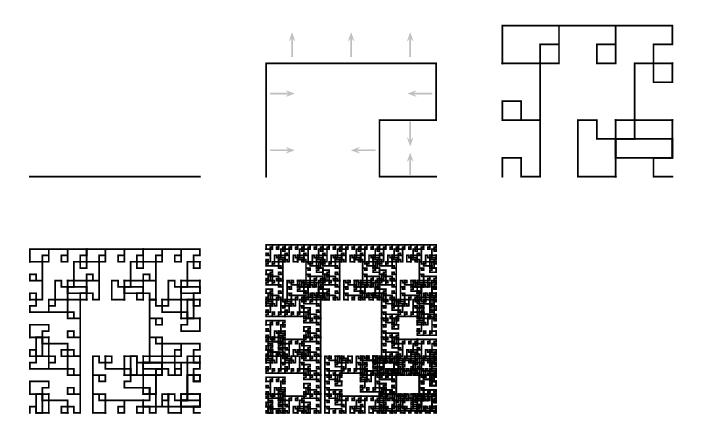}
\caption{The sets $I$ and $T^k(I)$ for $k=1,2,3,4$ of an IFS path $\Ss=\{S_1,...,S_9\}$ with the Sierpinski carpet as its invariant set. The gray arrows illustrate the orientation of the mappings $S_i$. Note that for this IFS path the similarity dimension equals $\frac{\log(9)}{\log(3)}=2$ where the Hausdorff dimension of the Sierpinski carpet is $\frac{\log(8)}{\log(3)}$.}
\label{pic_carpet}
\end{figure}
\end{center}

\bibliography{IFS_quasiarcs_biblio}
\bibliographystyle{abbrv}

\end{document}